\documentclass[11pt,reqno,a4paper]{amsart}
\usepackage{amsmath,amssymb,amsthm,amsaddr}

\usepackage[margin=3cm,top=4cm,bottom=3.5cm,footskip=1cm,headsep=1cm]{geometry}
\usepackage[utf8]{inputenc}

\def\RR{\mathbb{R}}
\def\div{\mathop{\rm div}}

\def\Th{\mathcal{T}_h}
\def\eps{\varepsilon}

\author{Herbert Egger}
\address{Department of Mathematics, TU Darmstadt, Germany}
\title[Energy norm estimates for parabolic problems]{Energy norm error estimates for finite element discretization of parabolic problems}

\newtheorem{lemma}{Lemma}
\newtheorem{problem}[lemma]{Problem}
\newtheorem{theorem}[lemma]{Theorem}

\begin{document}

\begin{abstract} 
We consider the discretization of parabolic initial boundary value problems by finite element methods in space and a Runge-Kutta time stepping scheme. Order optimal a-priori error estimates are derived in an energy-norm under natural smoothness assumptions on the solution and without artificial regularity conditions on the parameters and the domain. The key steps in our analysis are the careful treatment of time derivatives in the $H^{-1}$-norm and the the use of an $L^2$-projection in the error splitting instead of the Ritz projector. This allows us to restore the optimality of the estimates with respect to smoothness assumptions on the solution and to avoid artificial regularity requirements for the problem, usually needed for the analysis of the Ritz projector, which limit the applicability of previous work. The wider applicability of our results is illustrated for two irregular problems, for which previous results can either not by applied or 
yield highly sub-optimal estimates.
\end{abstract}

\maketitle

\begin{quote}
\end{quote}

\begin{quote}
\noindent
{\small {\bf AMS-classification:} 35K20, 65L06, 65L70, 65M15, 65M20, 65M60}
\end{quote}

\section{Introduction}

Let $\Omega \subset \RR^d$ be a bounded domain and $T>0$ denote a time horizon. 
We consider the numerical solution of parabolic initial boundary value problems of the form 
\begin{subequations} \label{eq:1}
\begin{align} 
&&&& u'(t)  + A(t) u(t) &= f(t),  && \text{in } \Omega,    \quad  0<t<T, &&&&\label{eq:1a}\\
&&&&               u(t) &= 0,     && \text{on } \partial\Omega, \ 0<t<T, &&&&\label{eq:1b}\\
&&&&               u(0) &= u_0,   && \text{in } \Omega. &&&&\label{eq:1c} 
\end{align}
\end{subequations}
Here $u'$ is the time derivative and $Au=-\div(a \nabla u) + b \cdot \nabla u + c u$ is a second order differential operator with coefficients $a,b,c \in L^\infty$ and $a(x,t) \ge \underline{a}>0$, such that $A(t)$ is uniformly elliptic for every point in time.  
It is well known, that for any $u_0 \in L^2(\Omega)$ and any $f \in L^2(0,T;H^{-1}(\Omega))$ there exists a unique weak solution in the \emph{energy-space}
\begin{align} \label{eq:W}
W(0,T) = \{ u \in L^2(0,T;H_0^1(\Omega) :  u' \in L^2(0,T;H^{-1}(\Omega))\},
\end{align}
which can be bounded a-priori by
\begin{align} \label{eq:apriori}
\|u\|_{W(0,T)} 
&=\|u\|_{L^2(0,T;H^1_0(\Omega))} + \| u'\|_{L^2(0,T;H^{-1}(\Omega))} \\
&\le C \big( \|f\|_{L^2(0,T;H^{-1}(\Omega))} + \|u_0\|_{L^2(\Omega)}\big).
\end{align}
Moreover, the constant $C$ only depends on the domain and the bounds for the coefficients.
We will refer to $\|\cdot\|_{W(0,T)}$ as the \emph{energy-norm} of the problem.
If the coefficients and the data are sufficiently smooth, 
the solution of \eqref{eq:1a}-\eqref{eq:1c} can be expected to be more regular:
For instance, if one assumes that $a,b,c \in W^{1,\infty}$, then 
\begin{align} \label{eq:reg1}
&\|u'\|_{L^2(0,T;H_0^1(\Omega))} + \|u''\|_{L^2(0,T;H^{-1}(\Omega))} \\
&\qquad \qquad \qquad \qquad 
\le C \big( \|f'\|_{L^2(0,T;H^{-1}(\Omega))} + \|f(0) - A(0) u_0\|_{L^2(\Omega)}\big),  \notag
\end{align}
whenever the right hand side is bounded. 
%
If, in addition, also the domain is sufficiently regular and $u_0 \in H_0^1(\Omega)$, then 
\begin{align} \label{eq:reg2}
&\|u\|_{L^2(0,T;H^2(\Omega))} + \|u'\|_{L^2(0,T;L^2(\Omega))} 
\le C \big( \|f\|_{L^2(0,T;L^2(\Omega))} + \|u_0\|_{H^1_0(\Omega)}\big). 
\end{align}
We refer to \cite{Evans} for details and proofs of these estimates and further results.

Let us emphasize at this point that the two estimates \eqref{eq:reg1} and \eqref{eq:reg2} also characterize the basic regularity spaces for parabolic problems. In particular, time derivatives typically lack two order of spatial regularity compared to the solution itself. 

\medskip 

In this paper, we study the numerical solution of \eqref{eq:1a}-\eqref{eq:1c}
by finite element discretization in space and Runge-Kutta time stepping schemes.
Such approximations have been investigated intensively in the literature, see e.g. 
\cite{Dupont74,DouglasDupont70,DouglasDuponWheeler,Thomee74,Varga,Wheeler73}; 
we refer to \cite{Thomee} for  a comprehensive collection of results and further references.
Our main goal here is to derive order optimal a-priori error estimates 
in the energy-norm $\| \cdot \|_{W(0,T)}$. These can be obtained
under minimal regularity assumptions on the coefficients and the domain, 
and we only require natural smoothness conditions for the solution.
To keep the notation simple, we consider here in detail only the lowest order approximation 
by piecewise linear finite elements in space and the backward Euler method in time. 
The generalization of our arguments to approximations of higher order is however straight forward.

To put our results into perspective, let us shortly recall some of the standard results for the Galerkin approximation of parabolic problems \cite{Thomee,Varga,Wheeler73} and compare them with the main contributions of our manuscript: 
 
\medskip 
 
As a first step, we will consider the semi-discretization in space 
by piecewise linear finite elements. 
The standard lowest order error estimate for parabolic problems then reads 
\begin{align} \label{eq:ee1}
\|u-u_h\|_{L^\infty(0,T;L^2(\Omega))} 
\le C h \big( \|u_0\|_{H^1(\Omega)} + \|u'\|_{L^1(0,T;H^1(\Omega))} \big).
\end{align}
Here $u_h$ is the finite element approximation and $h$ the meshsize.
This estimate amounts to \cite[Theorem~1.2]{Thomee} with $r=1$.
Assuming sufficient regularity of the domain, the parameters, and the solution, 
also second order convergence can be obtained here. 
The basic step in the proof of this error estimate is the decomposition of the error into
\begin{align} \label{eq:decomp}
u-u_h = (u-R_h u) + (R_h u - u_h), 
\end{align}
where $R_h$ denotes the Ritz projection associated to the operator $A$; see \cite{Varga,Wheeler73}.
The proof of the error bound then relies on a discrete energy estimate for the semi-discrete problem 
and certain properties of the Ritz projection, in particular, on a bound 
\begin{align} \label{eq:ritz}
\|u - R_h u\|_{L^2(\Omega)} \le C h \|u\|_{H^1(\Omega)},
\end{align} 
for the $L^2$-norm error, which can be obtained by the usual duality arguments \cite{Aubin67,Nitsche68}.
Without some additional regularity assumptions on the parameters and the domain, 
the estimate \eqref{eq:ritz} for the Ritz projection is however not valid, 
and therefore the validity of the bound \eqref{eq:ee1} for irregular problems cannot 
be granted by the proofs given in \cite{Thomee,Varga,Wheeler73}.
Also note that the estimate \eqref{eq:ee1} is somwhat sub-optimal concerning the regularity requirements 
for the solution: 
In fact, it should suffice to assume that $u \in L^\infty(0,T;H^1(\Omega))$, 
which is already valid if $u \in L^2(0,T;H^2(\Omega)) \cap H^1(0,T;L^2(\Omega))$. 
Let us emphasize that no spatial regularity for the time derivative seems to be required. 
Also, in view of the a-priori estimate \eqref{eq:reg2}, this latter condition would be a natural regularity assumption.

As a replacement for \eqref{eq:ee1}, we will derive an error estimate 
in the energy-norm $\|\cdot\|_{W(0,T)}$, which arises in the a-priori estimate for the solution. We will show that
\begin{align} \label{eq:res1}
\|u-u_h\|_{W(0,T)}
&=\|u-u_h\|_{L^2(0,T;H^1(\Omega))} + \| u' -  u'_h\|_{L^2(0,T;H^{-1}(\Omega))} \notag \\
&\le C h \big( \|u\|_{ L^2(0,T;H^2(\Omega))} + \| u'\|_{L^2(0,T;L^2(\Omega))} \big). 
\end{align}
By continuous embedding, one has $\|u-u_h\|_{L^\infty(0,T;\Omega)} \le C \|u - u_h\|_{W(0,T)}$, 
which yields the corresponding estimate also for the error in the $L^2$-norm at every point in time.
One can easily see that the estimate \eqref{eq:res1} is optimal with respect to the approximation properties of the finite element spaces and also with respect to the smoothness requirements on the solution, which are natural for the problem under investigation. 
Also note that we require no additional regularity of the domain or the parameters
for the proof of this result.

\medskip 

As a second step, we will then investigate the time discretization by the backward Euler  method. 
The standard error estimate for the fully discrete approximation reads
\begin{align} \label{eq:ee2}
\max_{0 \le t^n \le T}  \|u(t^n) - u_h^n\|_{L^2(\Omega)} \le C h \big( \|u_0\|_{H^1(\Omega)} + \|u_t\|_{L^1(0,T;H^1(\Omega))} + \tau \|u_{tt}\|_{L^1(0,T;L^2)}\big),
\end{align}
see \cite[Theorem~1.5]{Thomee} with $r=1$. 
Here $\tau = t^{n} - t^{n-1}$ is the time-step size, and $u_h^n$ denotes the $n$-th iterate of the Euler method. 
Again, we observe a certain sub-optimality concerning the regularity requirments for the solution. 
Moreover, the proof given in \cite{Thomee}, see also \cite{Wheeler73,Varga}, is only valid under additional 
restrictive regularity assumptions on the parameters and the domain, which strongly limit the applicability of the results. 

As a replacement for \eqref{eq:ee2}, we will derive the energy-norm estimate
\begin{align} \label{eq:res2}
\|u - \tilde u_h\|_{W(0,T)} 
&=\|u-\tilde u_h\|_{L^2(0,T;H^1(\Omega))} + \| u' - \tilde u_h'\|_{L^2(0,T;H^{-1}(\Omega))} \\
&\le C \big( 
  h \| u\|^2_{L^2(0,T;H^2(\Omega))} + h \|u'\|^2_{L^2(0,T;L^2(\Omega))}
\notag \\ & \qquad \qquad \qquad 
+ \tau \|u'\|_{L^2(0,T;H^1(\Omega))} + \tau \|u''\|_{L^2(0,T;H^{-1}(\Omega))}   
\big).\notag
\end{align}
Here $\tilde u_h$ denotes the function obtained from $u_h^n$ by piecewise linear interpolation in time. 
The bound for $u_h^n$ again follows easily by embedding of $W(0,T)$ into $L^\infty(0,T;L^2(\Omega))$.
As before, the estimate \eqref{eq:res2} can be seen to be optimal with respect to the approximation properties of the discretization, and, in view of the a-priori bounds \eqref{eq:reg1} and \eqref{eq:reg2}, the regularity requirements for the solution seem natural. 
Moreover, no artificial regularity of the domain and the coefficients will be required for the proof of this estimate.


\medskip

One key step in the derivation of our results will be the careful estimation of time derivatives in the $H^{-1}$-norm. 
This seems natural in view of the definition of the energy-norm and its role in the a-priori estimates. 
The importance of such estimates has already been observed in the context of a-posteriori error estimation \cite{MakridakisNochetto03}.
Let us also mention \cite{DouglasDupont77,Wheeler75}, where $H^{-1}$-Galerkin methods for the solution 
of parabolic initial boundary value problems have been investigated.

A second difference to previous investigations is, that we use here a somewhat different error splitting as usually employed, namely 
\begin{align}
u - u_h = (u - \pi_h u) + (\pi_h u - u_h), 
\end{align}
where $\pi_h$ is the $L^2$-projection onto the finite element space.
The $L^2$-projection $\pi_h$ has important advantages in comparison to the Ritz-projector $R_h$:
First, the estimate corresponding to \eqref{eq:ritz} and further approximation properties can be proven without 
regularity assumptions on the domain or the parameters. 
In addition, $\pi_h$ commutes with the time derivative, even for problems with time dependent parameters. 
At several places in our analysis, we will rely on the $H^1$-stability of the $L^2$-projection \cite{BankYserentant14,BramblePasciakSteinbach02}. 
Morover, we use approximation error estimates for the $L^2$-projection in various norms, in particular including estimates in the $H^{-1}$-norm.

\medskip

The remainder of the manuscript is organized as follows:
In Section~\ref{sec:prelim}, we formally present the problem to be investigated 
and we discuss the basic assumptions on the domain and the coefficients. 
%
%
In Section~\ref{sec:fem}, we introduce the finite element spaces and summarize
some estimates for the $L^2$-projection that are required later in our analysis.  
Section~\ref{sec:semi} is then concerned with the derivation of the estimate \eqref{eq:res1} for the semi-discretization. 
In Section~\ref{sec:time}, we investigate the time discretization by the backward Euler method and we derive the second estimate \eqref{eq:res2}. 
For illustration of the wider applicability of our results, 
we discuss in Section~\ref{sec:num} two irregular test proplems, 
for which, due to lack of regularity, the standard bounds \eqref{eq:ee1} and \eqref{eq:ee2} cannot be 
applied directly.
Due to the weaker requirements for our estimates \eqref{eq:res1} and \eqref{eq:res2}, 
the optimal convergence in the energy-norm can however still be guaranteed also theoretically.
We close with a short discussion of our results and highlight some possibilities for generalization and future investigations.

\section{Preliminaries} \label{sec:prelim}


Let $\Omega \subset \RR^d$, $d=2,3$, be some bounded domain and $T>0$.
We use standard notation for function spaces, see e.g. \cite{Evans};
in particular,
$H_0^1(\Omega)$ denotes the sub-space of functions in $H^1(\Omega)$ with vanishing traces on $\partial \Omega$,
and $H^{-1}(\Omega) = (H_0^1(\Omega))'$ is the space of bounded linear functional on $H^1_0(\Omega)$; 
the corresponding duality product is denoted by $\langle \cdot, \cdot\rangle_{H^{-1}(\Omega) \times H_0^1(\Omega)}$.
%
%
%
Of particular importance for our analysis is the energy-space
\begin{align}
W(0,T) = \{u \in L^2(0,T;H_0^1(\Omega)) : u' \in L^2(0,T;H^{-1}(\Omega))\},
\end{align}
which is equipped with the norm $\|u\|_{W(0,T)}=\|u\|_{L^2(0,T;H^1(\Omega))} + \|u'\|_{L^2(0,T;H^{-1}(\Omega))}$.
For later reference, let us recall the following embedding result \cite[Sec.~5.9]{Evans}.
\begin{lemma} \label{lem:embedding}
Let $u \in W(0,T)$. Then $u \in C([0,T];L^2(\Omega))$ and 
$$
\sup_{0 \le t \le T}\|u(t)\|_{L^2(\Omega)} \le C \|u\|_{W(0,T)}.
$$
\end{lemma}
As a consequence, all estimates derived in the energy-norm $\|\cdot\|_{W(0,T)}$ 
automatically yield corresponding bounds in $\|\cdot\|_{L^\infty(0,T;L^2(\Omega))}$, 
i.e., pointwise in time. Here and below, we denote by $C$ some generic constant which may have different values in different occasions. 

\medskip

Let us now turn to the initial-boundary value problem \eqref{eq:1a}-\eqref{eq:1c}. 
We assume that the operator $A$ has the form $A u = -\div (a \nabla u) + b \cdot \nabla u + c u$.
In order to proof well-posedness of the initial boundary value problem, we assume 
that the parameters satisfy 
\begin{itemize}
 \item[(A1)] $a,c \in L^\infty(\Omega \times (0,T))$ and $b \in L^\infty(\Omega \times (0,T))^d$, and
 \item[(A2)] $a(x,t) \ge \underline a$ for some constant $\underline a > 0$ and a.e. $(x,t) \in \Omega \times (0,T)$.
\end{itemize}
Because of the first assumption, $A$ defines a bounded linear operator from $L^2(0,T;H_0^1(\Omega))$ to $L^2(0,T;H^{-1}(\Omega))$.
For a.e. $t \in (0,T)$ we then define an associated bilinear form by
\begin{align} 
a(u,v;t)  = \int_\Omega a(x,t) \nabla u(x) \nabla v(x) + b(x,t) \cdot \nabla u(x) v(x) + c(x,t) u(x) v(x) \ dx.
\end{align}
The weak formulation of the initial boundary value problem \eqref{eq:1a}-\eqref{eq:1c} now reads
\begin{problem}[Weak formulation] $ $ \label{prob:weak}\\
Given $f \in L^2(0,T;H^{-1}(\Omega))$ and $u_0 \in L^2(\Omega)$,
find $u \in W(0,T)$ such that $u(0) = u_0$ and 
\begin{align*} 
\langle u'(t), v \rangle_{H^{-1}(\Omega) \times H_0^1(\Omega)} + a(u(t),v;t) &= \langle f(t) , v\rangle_{H^{-1}(\Omega) \times H_0^1(\Omega)}
\end{align*}
holds for all test functions $v \in H_0^1(\Omega)$ and for a.e. $t \in (0,T)$.
\end{problem}
\noindent 
Under assumptions (A1)-(A2),  Problem~\ref{prob:weak} has a unique solution $u \in W(0,T)$ 
and there holds $\|u\|_{W(0,T)} \le C \big( \|u_0\|_{L^2(\Omega)} + \|f\|_{L^2(0,T;H^{-1}(\Omega)} \big)$ 
with a constant $C$ that depends only on the bounds for the coefficients and on the domain; see \cite{Evans} for details. 
The proof relies on a Gronwall argument and the following properties of the bilinear form. 
\begin{lemma}[Continuity and G{\aa}rding inequality] $ $ \label{lem:garding}\\
Let (A1) and (A2) hold. 
Then there exist constants $C_a,\alpha,\eta > 0$ such that 
\begin{align} \label{eq:garding}
a(u,v;t) \le C_a \|u\|_{H^1(\Omega)} \|v\|_{H^1(\Omega)} 
\quad \text{and} \quad 
a(u,u;t) + \eta \|u\|^2_{L^2(\Omega)} \ge  \alpha \|u\|_{H^1(\Omega)}^2,
\end{align} 
and these estimates hold uniformly for a.e. $t \in (0,T)$ and all $u,v \in H^1(\Omega)$. 
\end{lemma}
\begin{proof}
The continuity follows from the Cauchy-Schwarz inequality. 
Using the lower and upper bounds for the coefficients, we get 
\begin{align*} 
a(u,u;t) \ge \underline a \|\nabla u\|^2_{L^2(\Omega)} + \|b(t)\|_{L^\infty} \|\nabla u\|_{L^2(\Omega)} \|u\|_{L^2(\Omega)} + \|c(t)\|_{L^\infty} \|u\|_{L^2(\Omega)}^2.  
\end{align*}
The G{\aa}rding inequality then follows by Young's inequality and choosing the coefficients, for instance, as $\alpha = \underline a/2$ and $\eta=\underline{a}/2 + \|b\|_{L^\infty}^2/(2 \underline a) + \|c\|_{L^\infty}$.
\end{proof}

\section{Properties of the $L^2$-projection onto finite element spaces} \label{sec:fem}

For the semi-discretization in space, we will employ a standard finite element method. 
To avoid technical difficulties, we assume that $\Omega$ is polyhedral and that it can be partitioned into a set $\Th=\{T\}$ of simplicial elements $T$. More precisely, we require that 
\begin{itemize}
 \item[(A3)] $\Th$ is a regular simplicial partition of $\Omega$, 
             i.e., the intersection of two different elements either empty, or a vertex, 
             an entire edge, (an entire face) of both elements;
 \item[(A4)] $\Th$ is locally quasi-uniform, i.e., there exists a $\gamma>0$ such that 
             the diameter $h_T$ of an element $T$ and the radius $\rho_T$ of the largest ball that can be inscribed in $T$ 
             are related by $\gamma h_T \le \rho_T \le h_T$ for all elements $T$.
\end{itemize}
Given such a mesh $\Th$, we consider the standard finite element space 
\begin{align}
V_h = \{v \in H_0^1(\Omega) : v|_{T} \in P_1(T) \text{ for all } T \in \Th\},
\end{align}
of piecewise linear continuous functions that vanish on the boundary. 
Furthermore, we denote by $\pi_h : L^2(\Omega) \to V_h$ the $L^2$-orthogonal projection defined by 
\begin{align*}
(\pi_h u, v_h)_{L^2(\Omega)} = (u, v_h)_{L^2(\Omega)} \qquad \text{for all } v_h \in V_h.
\end{align*}
Obviously, $\|\pi_h u\|_{L^2(\Omega)} \le \|u\|_{L^2(\Omega)}$, i.e., the $L^2$-projection is 
stable (a bounded linear operator) on $L^2(\Omega)$. 
Under the assumption (A3)-(A4), it is however also stable on $H^1_0(\Omega)$.
\begin{lemma}[$H^1$-stability] \label{lem:h1stability}
Let (A3)-(A4) hold. Then 
\begin{align*}
\|\pi_h u\|_{H^1(\Omega)} \le C \|u\|_{H^1(\Omega)} \qquad \text{for all } u \in H_0^1(\Omega),
\end{align*}
and the constant $C$ only depends on the domain and the regularity constants of the mesh.
\end{lemma}
\noindent 
For globally quasi-uniform meshes, the result follows directly from the Bramble-Hilbert Lemma and 
an inverse inequality. The proof for locally quasi-uniform meshes has been given in \cite{BramblePasciakSteinbach02}; 
see \cite{BankYserentant14} for generalizations including higher order approximations.

\medskip 

\noindent 
We will also require the following approximation error estimates. 
\begin{lemma} \label{lem:approximation}
Let (A3)-(A4) hold. Then 
\begin{align*}
\|u - \pi_h u\|_{H^s(\Omega)} \le C h^{k-s} \|u\|_{H^{k}(\Omega)} 
\end{align*}
for all $u \in H_0^{k}$ with $0 \le k \le 2$ and $-1 \le s \le \min\{1,k\}$. 
\end{lemma}
\begin{proof}
For completeness, we sketch the main steps.
The case $s=0$ and $0 \le k \le 2$ is well known and follows from the Bramble-Hilbert lemma and scaling arguments. 
To show the estimate for $s=1$ and $1 \le k \le 2$, 
let us denote by $\pi^1_h : H_0^1(\Omega) \to V_h$ the $H^1$-orthogonal projection defined by 
\begin{align*}
(\pi_1^h u, v_h )_{H^1(\Omega)} = (u,v_h)_{H^1(\Omega)} \qquad \text{for all } v_h \in V_h.
\end{align*}
Recall that $\|\pi^1_h u - u\|_{H^1(\Omega)} \le C' h^{k-1} \|u\|_{H^k(\Omega)}$ for $1 \le k \le 2$, 
which is the usual finite element error estimate \cite{BrennerScott}. 
We can then proceed by
\begin{align*}
\|u - \pi_h u \|_{H^1(\Omega)} 
&\le \|u - \pi_h \pi_h^1 u\|_{H^1(\Omega)} + \|\pi_h (u - \pi_h^1 u)\|_{H^1(\Omega)} \\
&\le (1 + C) \|u - \pi^1_h u\|_{H^1(\Omega)} \le (1+C) C' h \|u\|_{H^2(\Omega)},
\end{align*}
where we used the projection property and the $H^1$-stability of $\pi_h$, and the approximation properties of $\pi^1_h$ in the last two steps.
Now assume that $u \in H^k(\Omega)$ with $0 \le k \le 2$. Then
\begin{align*}
\|u - \pi_h u\|_{H^{-1}(\Omega)} 
&= \sup_{v \in H_0^1(\Omega)} (u - \pi_h u, v)_{L^2(\Omega)} /\|v\|_{H^1(\Omega)} \\
&= \sup_{v \in H_0^1(\Omega)} (u, v - \pi_h v)_{L^2(\Omega)} /\|v\|_{H^1(\Omega)}
\le C h^{k+1} \|u\|_{H^k(\Omega)}.
\end{align*}
Here we used the approximation property of $\pi_h$ for $s=0$. 
This yields the estimate for $s=-1$ and $0 \le k \le 2$ and completes the proof.
\end{proof}
Corresponding estimates for real valued $s$ and $k$ can be obtained by interpolation arguments. 
Also note that the dependence on $h$ in the above estimates can be localized.

\section{Galerkin semi-discretization} \label{sec:semi}

As a first step in the approximation process, let us investigate the semi-discretization in space by finite elements.
Proceeding in a standard fashion, we define
\begin{problem}[Semi-discretization] $ $ \label{prob:semi} \\
Find $u_h \in H^1(0,T;V_h)$ such that $u_h(0)=\pi_h u_0$ and 
\begin{align} \label{eq:semi}
(u_h'(t), v_h)_{L^2(\Omega)} + a(u_h(t),v_h;t) &= \langle f(t), v_h \rangle_{H^{-1}(\Omega) \times H_0^1(\Omega)}  
\end{align}
for all test functions $v_h \in V_h$ and a.e. $t \in (0,T)$.
\end{problem}
The first term could also be written as $\langle u_h'(t), v_h \rangle_{H^{-1}(\Omega) \times H_0^1(\Omega)}$.
By choosing a basis for $v_h$, the semi-discrete problem yields an ordinary differential equation, 
and existence and uniqueness follow by the Picard-Lindelöf theorem. 
More precicely, we have
\begin{lemma}[Well-posedness and discrete a-priori estimate] $ $ \label{lem:discrete_apriori}\\
Let (A1)--(A4) hold. Then Problem~\ref{prob:semi} has a unique solution $u_h \in H^1(0,T;V_h)$ and 
\begin{align*}
\|u_h\|_{W(0,T)} \le C (\|u_0\|_{L^2(\Omega)} + \|f\|_{L^2(0,T;H^{-1}(\Omega))}).
\end{align*}
The constant $C$ in this estimate depends only on the bounds for the coefficients, the domain, and the constants characterizing the regularity of the mesh. 
\end{lemma}
\begin{proof}
It remains to verify the a-priori estimate. 
By testing the discrete variational problem in the usual way with $v_h = u_h(t)$, one can show that 
\begin{align*}
\frac{1}{2} \frac{d}{dt}\|u_h(t)\|_{L^2(\Omega)}^2 + \frac{\alpha}{2} \|u_h(t)\|^2_{H^1(\Omega)}
&\le \eta \|u\|^2_{L^2(\Omega)} + \frac{1}{2\alpha}\|f(t)\|_{H^{-1}(\Omega)}^2.
\end{align*}
Here we only used the G{\aa}rding inequality \eqref{eq:garding} and some elementary manipulations.
Integrating over time and applying a Gronwall argument then yields the energy estimate
\begin{align*}
\|u\|_{L^\infty(0,T;L^2(\Omega))}^2 +\|u\|_{L^2(0,T;H^1(\Omega))}^2 
\le C \big(\|u_0\|_{L^2(\Omega)}^2 +  \|f\|_{L^2(0,T;H^{-1}(\Omega))}^2 \big).  
\end{align*}
To obtain the remaining bound for the time derivative, observe that the $H^{-1}$-norm 
of a discrete function can be expressed as
\begin{align*}
\|u_h'(t)\|_{H^{-1}(\Omega)} 
&= \sup_{v \in H_0^1(\Omega)} (u_h'(t), v)_{L^2(\Omega)} / \|v\|_{H^1(\Omega)}   \\
&= \sup_{v \in H_0^1(\Omega)} (u_h'(t), \pi_h v)_{L^2(\Omega)} / \|v\|_{H^1(\Omega)}. 
\end{align*}
Using the discrete variational problem with $v_h = \pi_h v$ further yields 
\begin{align*}
(u_h'(t), \pi_h v)_{L^2(\Omega)} 
&= \langle f_h(t), \pi_h v\rangle_{H^{-1}(\Omega) \times H_0^1(\Omega)} - a(u_h(t),\pi_h v;t) \\
&\le \big( \|f_h(t)\|_{H^{-1}(\Omega)} + C \|u_h(t)\|_{H^1(\Omega)} \big) \|\pi_h v \|_{H^1(\Omega)}.  
\end{align*}
Since $\pi_h$ is bounded on $H_0^1(\Omega)$, we further have $\|\pi_h v\|_{H^1(\Omega)} \le C \|v\|_{H^1(\Omega)}$.
Using these estimates in the expression for the norm of the time derivative finally yields 
\begin{align*}
\|u_h(t)\|_{H^{-1}(\Omega)} \le C' \big( \|f_h(t)\|_{H^{-1}(\Omega)} + \|u_h(t)\|_{H^1(\Omega)} \big).
\end{align*}
The result then follows by integration over time and the energy estimate derived before.
\end{proof}

As a direct consequence of the definition and the similarity of the continuous and the semi-discrete variational problems, 
we obtain 
\begin{lemma}[Galerkin orthogonality] $ $ \label{lem:galerkin_orthogonality}\\
Let $u$ and $u_h$  denote the solutions of Problems~\ref{prob:weak} and \ref{prob:semi}, respectively. 
Then 
\begin{align*}
\langle u(t) - u_h(t), v_h\rangle_{H^{-1}(\Omega) \times H_0^1(\Omega)} 
+ a(u(t)-u_h(t), v_h;t) = 0
\end{align*}
for almost every $t \in (0,T)$ and all $v_h \in V_h$. Moreover, $\pi_h (u(0)-u_h(0))=0$.
\end{lemma}

\noindent 
We can now turn to the error analysis of the semi-discretization. 
To this end, we divide the error into an approximation error and a discrete error as
\begin{align*}
u(t) - u_h(t) = (u(t) - \pi_h u(t)) + (\pi_h u(t) - u_h(t)) = (i) + (ii) . 
\end{align*}
Due to the approximation properties of $\pi_h$, the first term can be bounded readily by 
\begin{lemma}[Approximation error] \label{lem:semi_approximation_error}
Let (A3)-(A4) hold. Then 
\begin{align*}
\|u - \pi_h u\|_{W(0,T)} \le C h \big( \|u\|_{L^2(0,T;H^2(\Omega)} + \|u'\|_{L^2(0,T;L^2(\Omega))}\big),
\end{align*}
and $C$ only depends on the domain and the shape-regularity constants of the mesh.
\end{lemma}
Using Galerkin orthogonality and the discrete stability of the method, 
the discrete error component can then be bounded as usual by the approximation error as well.
\begin{lemma}[Discrete error] \label{lem:semi_discrete_error}
Let (A1)--(A4) hold. Then
\begin{align*}
\|\pi_h u - u_h\|_{W(0,T)} \le C \|u - \pi_h u\|_{W(0,T)}.
\end{align*}
Moreover, the constant $C$ only depends on the domain, 
the bounds for the coefficients, and the shape-regularity of the mesh.
\end{lemma}
\begin{proof}
The discrete error $e_h = \pi_h u(t) - u_h(t)$ satisfies 
\begin{align*}
(e_h(t), v_h)_{L^2(\Omega)} + a(e_h(t),v_h; t) 
&=(\pi_h u'(t) - u'(t), v_h)_{L^2(\Omega)} + a(\pi_h u(t) - u(t), v_h;t)\\ 
&=a(\pi_h u(t) - u(t), v_h;t) =: \langle \tilde f(t), v_h\rangle_{H^{-1}(\Omega) \times H^1_0(\Omega)}, 
\end{align*}
and also $e_h(0) = \pi_h u(t) - u_h(t) = 0$. 
Using the continuity of the bilinear form, we have 
\begin{align*}
\|\tilde f(t)\|_{H^{-1}(\Omega)} \le C_a \|u(t) - \pi_h u(t)\|_{H^1(\Omega)}.
\end{align*}
By application of Lemma~\ref{lem:discrete_apriori} to the equation for the discrete error, we thus obtain
\begin{align*}
\|e_h\|_{W(0,T)} \le C \|\tilde f\|_{L^2(0,T;H^{-1}(\Omega))} \le  C' \|u - \pi_h u\|_{W(0,T)},
\end{align*}
and this  already proves the assertion of the lemma.
\end{proof}

\noindent 
By combining the previous two lemmas, we readily obtain our first main result. 
\begin{theorem}[Energy-norm error estimate for the semi-discretization] $ $ \label{lem:thm1}\\
Let (A1)-(A4) hold, and let $u$ and $u_h$ denote the solutions of Problems~\ref{prob:weak} and \ref{prob:semi}, respectively. 
Then 
\begin{align*}
\|u - u_h\|_{W(0,T)} \le C h \big( \|u\|_{L^2(0,T;H^2(\Omega)} + \|u'\|_{L^2(0,T;L^2(\Omega))}\big),
\end{align*}
and the constant $C$ only depends on the domain, the bounds for the parameters, and the shape regularity of the mesh.
\end{theorem}

Let us emphasize that the regularity requirements for the solution are natural. 
In view of approximation properties of finite elements, they are also almost necessary to 
guarantee the approximation order one in the energy-norm. 
Also note that no additional regularity of the domain or the coefficients was required for our proofs.

\section{Time stepping} \label{sec:time}

Let us now turn to the time discretization of the semi-discrete Problem~\ref{prob:semi} by the backward Euler method. 
For ease of presentation, we only consider here uniform time steps of size $\tau=T/N$, 
and therefore set $t^{n} = n \tau$. 
To allow for evaluation of the coefficients and the data at individual points in time, we further assume that 
\begin{itemize}
 \item[(A5)] $a,b,c \in W^{1,\infty}(0,T;L^\infty(\Omega))$, and
 \item[(A6)] $f \in H^1(0,T;H^{-1}(\Omega))$.
\end{itemize}
The condition (A5) could be relaxed by using a Discontinuous-Galerkin method for the time discretization. 
In addition, we require that the time step is sufficiently small, i.e.,  
\begin{itemize}
 \item[(A7)] $\tau < 1/\eta$,
\end{itemize}
where $\eta$ is the constant from the G{\aa}rding inequality \eqref{eq:garding}.
This condition could be avoided by treating the lower order terms in 
the bilinear form in an explicit manner. 
To facilitate the notation, we will use in the sequel
$$
\partial_\tau u_h^{n+1} = \frac{1}{\tau} (u_h^{n+1} - u_h^n)
$$ 
to denote the discrete time derivatives at $t^{n+1}$. 
Applying the backward Euler scheme for the time discretization of Problem~\ref{prob:semi} 
then leads to
\begin{problem}[Fully discrete scheme] $ $ \label{prob:full}\\
Set $u_h^0 := \pi_h u_0$ and find $u_h^n \in V_h$ for $1 \le n \le N$, such that
\begin{align} \label{eq:full}
(\partial_\tau u_h^{n+1}, v_h)_{L^2(\Omega)} + a(u_h^{n+1},v_h;t^{n+1}) = \langle f(t^{n+1}), v_h \rangle_{H^{-1}(\Omega) \times H_0^1(\Omega)}
\end{align}
holds for all test functions $v_h \in V_h$.
\end{problem}
\noindent
The equation \eqref{eq:full} amounts to an implicit time-stepping scheme and that the snapshots $u_h^n$ can be computed recursively. 
The assumption (A7) guarantees that the elliptic problem for each time-step is uniquely solvable.
More precisely, we have
\begin{lemma}[Discrete well-posedness and a-priori estimates] $ $ \label{lem:full_apriori}\\
Let (A1)-(A7) hold. 
Then Problem~\ref{prob:full} has a unique solution $\{u^n_h\}_{0 \le n \le N}$,
and 
\begin{align*}
\sum\nolimits_{n=1}^N \tau \big( \|\partial_\tau u_h^{n}\|_{H^{-1}(\Omega)}^2 +  \|u_h^n\|_{H^1(\Omega)}^2 \big) 
\le C \big( \|u_0\|_{L^2(\Omega)} + \|f\|_{H^1(0,T;H^{-1}(\Omega))}\big). 
\end{align*}
\end{lemma}
\noindent
The norm on the left hand side is a discrete version of the energy-norm $\|\cdot\|_{W(0,T)}$. 
\begin{proof}
Testing \eqref{eq:full} with $u_h^{n+1}$ and proceeding as in the proof of Lemma~\ref{lem:discrete_apriori} yields
\begin{align*}
&\frac{1}{2\tau} \|u_h^{n+1}\|_{L^2(\Omega)}^2 + \frac{\alpha}{2} \|u_h^{n+1}\|_{H^1(\Omega)}^2 \\
&\qquad \qquad 
\le \frac{1}{2\tau} \|u_h^{n}\|_{L^2(\Omega)}^2  + \eta \|u_h^{n+1}\|_{L^2(\Omega)}^2 
  + \frac{1}{2\alpha} \|f(t^{n+1})\|_{H^{-1}(\Omega)}^2.
\end{align*}
Via a Gronwall argument, we then obtain the discrete energy estimate
\begin{align*}
\max\nolimits_{1 \le n \le N} \|u_h^n\|^2_{L^2(\Omega)} + \sum\nolimits_n \tau \|u_h^n\|^2_{H^1(\Omega)} 
\le C \big(\|u_0\|_{H^1(\Omega)}^2 + \sum\nolimits_n \tau \|f(t^n)\|_{H^{-1}(\Omega))}^2\big)
\end{align*}
Due to condition (A5), the last term can be estimated via Talyor expansion by 
\begin{align*} 
\sum\nolimits_n \tau \|f(t^{n+1})\|_{H^{-1}(\Omega)}^2 
\le C \big( \|f\|_{L^2(0,T;H^{-1}(\Omega))}^2 + \tau^2 \|f'\|^2_{L^2(0,T;H^{-1}(\Omega))} \big).
\end{align*}
In order to derive the estimate for $\partial_\tau u_h^n$, recall that 
\begin{align*}
\|\partial_\tau u_h^{n+1}\|_{H^{-1}(\Omega)} 
= \sup_{v \in H_0^1(\Omega)} (\partial_\tau u_h^{n+1}, \pi_h v)_{L^2(\Omega)}/\|v\|_{H^1(\Omega)}.
\end{align*}
Using $v_h = \pi_h v$ as a test function in the discrete scheme \eqref{eq:full}, we further obtain 
\begin{align*}
(\partial_\tau u_h^{n+1}, \pi_h v)_{L^2(\Omega)}
&= \langle f(t^{n+1}), \pi_h v \rangle_{H^{-1}(\Omega) \times H_0^1(\Omega)} - a(u_h^{n+1},\pi_h v;t^{n+1}) \\
&\le C (\|f(t^{n+1})\|_{H^{-1}(\Omega)} + C \|u_h^{n+1}\|_{H^1(\Omega)} \big) \|\pi_h v\|_{H^1(\Omega)}.
\end{align*}
Using the stability of the projection $\pi_h$, this allows to estimate the discrete time derivative by known terms.
The assertion then follows from the previous estimates. 
\end{proof}



%

Let us now turn to the derivation of error estimates for the full discretization defined in Problem~\ref{prob:full}.
Similar as in the previous section, we use an error decomposition 
into an approximation error and a discrete error by 
\begin{align*}
u(t^n) - u_h^n = (u(t^n) - \pi_h u(t^n)) + (\pi_h u(t^n) - u_h^n) = (i) + (ii).
\end{align*}
The first component can be bounded by the approximation estimates for $\pi_h$ as follows.
\begin{lemma}[Approximation error] \label{lem:full_approx_error}
Let (A3)-(A4) hold. Then 
\begin{align*}
&\sum\nolimits_{n=1}^N \tau \big( \|\partial_\tau u(t^n) - \partial_\tau \pi_h u(t^n)\|_{H^{-1}(\Omega)}^2 
    + \|u(t^n) - \pi_h u(t^n)\|_{H^1(\Omega)}^2\big) \\
&\qquad \qquad \le 
C \big( 
h^2 \| u\|^2_{L^2(0,T;H^2(\Omega))} +
h^2 \| u'\|^2_{L^2(0,T;L^2(\Omega))} +
\tau^2 \|u'\|^2_{L^2(0,T;H^1(\Omega))} 
\big).
\end{align*}
\end{lemma}
\begin{proof}
Using Taylor expansion in time and the properties of the projection, we obtain 
\begin{align*}
\sum\nolimits_n \tau &\|\partial_\tau u(t^n) - \partial_\tau \pi_h u(t^n)\|^2_{H^1(\Omega)} \\
&\le C \|u' - \pi_h u'\|_{L^2(0,T;H^{-1}(\Omega))}^2 
\le C' h^2 \|u'\|_{L^2(0,T;L^2(\Omega))}^2.
\end{align*}
In a similar manner, we obtain for the second term
\begin{align*}
\sum\nolimits_n \tau &\|u(t^n) - \pi_h u(t^n)\|^2_{H^1(\Omega)} \\
&\le C \big( \|u - \pi_h u\|_{L^2(0,T;H^1(\Omega))}^2 + \tau^2 \|u' - \pi_h u'\|_{L^2(0,T;H^1(\Omega))}^2 \big) \\
&\le C' \big( h^2 \|u\|_{L^2(0,T;H^2(\Omega))} + \tau^2 \|u'\|_{L^2(0,T;H^1(\Omega))}^2 \big).
\end{align*}
The assertion then follows by summing up the two contributions.
\end{proof}

\noindent
The discrete stability of the scheme \eqref{eq:full} allows us to bound the discrete error as follows.
\begin{lemma}[Discrete error] \label{lem:full_discrete_error}
Let (A1)-(A7) hold. Then 
\begin{align*}
&\sum\nolimits_{n=1}^N \tau \big( \|\partial_\tau \pi_h u(t^n) - \partial_\tau u_h^n \|_{H^{-1}(\Omega)}^2 
    + \|\pi_h u(t^n) - u_h^n\|_{H^1(\Omega)}^2\big) \\
&\qquad \qquad \le 
C \big( 
h^2 \| u\|^2_{L^2(0,T;H^2(\Omega))} + 
\tau^2 \|u'\|^2_{L^2(0,T;H^1(\Omega))} + 
\tau^2 \|u''\|^2_{L^2(0,T;H^{-1}(\Omega))} 
\big).
\end{align*}
\end{lemma}
\begin{proof}
Let $e_h^n = \partial_\tau \pi_h u(t^{n}) - u_h^n$ denote the discrete error. Then 
\begin{align*} 
&\partial_\tau (e_h^{n+1},v_h)_{L^2(\Omega)} + a(e_h^{n+1},v_h;t^{n+1}) \\
& \qquad = (\partial_\tau u(t^{n+1}) - u'(t^{n+1}),v_h)_{L^2(\Omega)} 
   + a(\pi_h u(t^{n+1})-u(t^{n+1}),v_h;t^{n+1}). 
\end{align*}
With similar arguments as in the proof of Lemma~\ref{lem:full_apriori}, we then obtain
\begin{align*}
&\max_{1 \le n \le N}  \|e_h^n\|_{L^2(\Omega)}^2 + \sum\nolimits_n \tau \|e_h^n\|_{H^1(\Omega)}^2  \\
&\qquad \le  C \sum\nolimits_n \tau \big( \|\partial_\tau u(t^n) - u'(t^n)\|_{H^{-1}(\Omega)}^2 
  + \|\pi_h u(t^n) - u(t^n)\|_{H^1(\Omega)}^2\big).
\end{align*}
Using Taylor expansion, the first term on the right hand side can be further bounded by
\begin{align*}
\sum\nolimits_n \tau \|\partial_\tau u(t^n) - u'(t^n)\|_{H^{-1}(\Omega)}^2
\le \tfrac{\tau^2}{2} \|u''\|_{L^2(0,T;H^{-1}(\Omega))}^2, 
\end{align*}
In a similar way, we obtain for the second term 
\begin{align*}
&\sum\nolimits_n \tau \|\pi_h u(t^{n+1}) - u(t^{n+1})\|_{H^1(\Omega)}^2\\ 
&\qquad \le C \big(\|\pi_h u - u\|^2_{L^2(0,T;H^1(\Omega))} + \tau^2 \|\pi_h u' - u'\|^2_{L^2(0,T;H^1(\Omega))} \big) \\
&\qquad \le C' \big(h^2 \| u\|^2_{L^2(t^n,t^{n+1};H^2(\Omega))} +  \tau^2 \|u'\|^2_{L^2(t^n,t^{n+1};H^1(\Omega))}\big).
\end{align*}
Here we employed  Taylor expansions and the approximation and stability properties of the projection $\pi_h$. 
By combination of the two estimates we obtain 
\begin{align*}
&\max_{1 \le n \le N}  \|\pi_h u(t^n) - u_h^n\|_{L^2(\Omega)}^2 + \sum\nolimits_n \tau \|\pi_h u(t^n) - u_h^n\|_{H^1(\Omega)}^2  \\
&\qquad \qquad 
\le  C \big( 
h^2 \| u\|^2_{L^2(0,T;H^2(\Omega))} + 
\tau^2 \|u'\|^2_{L^2(0,T;H^1(\Omega))} + 
\tau^2 \|u''\|^2_{L^2(0,T;H^{-1}(\Omega))} 
\big).
\end{align*}
Using the characterization of the $H^{-1}$-norm for finite element functions, 
the definition of the discrete scheme \eqref{eq:full}, 
Galerkin-orthogonality, and some basic manipulations,
we can further bound the discrete time derivative terms by 
\begin{align*}
&\sum\nolimits_n \tau \|\partial_\tau \pi_h u(t^n) - u_h^n\|^2_{H^{-1}(\Omega)} 
\\ &\qquad 
\le C \sum\nolimits_n 
\tau \big( 
\|\partial_\tau u(t^{n}) - u'(t^n)\|_{H^{-1}(\Omega)}^2 
\\ &\qquad \qquad \qquad  \quad 
+ \|u(t^n) - \pi_h u(t^n)\|_{H^1(\Omega)}^2 
+\|u_h^n - \pi_h u(t^n)\|_{H^1(\Omega)} 
\big). 
\end{align*}
The first and second term on the right hand side can be estimated as above, 
and the third term is an approximation error which has already been bounded.
\end{proof}
Summing up the two estimates for approximation error and the discrete error yields
\begin{lemma}[Discrete energy-norm error estimate] \label{lem:full_discrete_energy}
Let (A1)-(A7) hold. Then 
\begin{align*}
&\sum\nolimits_{n=1}^N \tau \big( \|\partial_\tau u(t^{n+1}) - \partial_\tau u_h^{n+1} \|_{H^{-1}(\Omega)}^2 + \|u(t^{n+1}) - u_h^n\|_{H^1(\Omega)}^2\big) \\
& \qquad \qquad \le 
C \big( h^2 \| u\|^2_{L^2(0,T;H^2(\Omega))} + h^2 \|u'\|_{L^2(0,T;L^2(\Omega))} 
\\ & \qquad \qquad \qquad \qquad \qquad \qquad \qquad 
+ \tau^2 \|u'\|^2_{L^2(0,T;H^1(\Omega))} + \tau^2 \|u''\|^2_{L^2(0,T;H^{-1}(\Omega))}\big).
\end{align*}
\end{lemma}

\noindent
To finally obtain an estimate in the energy-norm $\|\cdot\|_{W(0,T)}$, let us denote by 
\begin{align} \label{eq:uhtilde}
\tilde u_h(t) = \frac{t^{n+1}-t}{t^{n+1}-t^n} u_h^{n} + \frac{t-t^n}{t^{n+1}-t^n} u_h^{n+1}, \qquad t^n \le t  \le t^{n+1}, 
\end{align}
the linear interpolant of the fully discrete approximations $u_h^n$ in time.
Using the definition of $\tilde u_h$ and simple manipulations, 
we now obtain the second main result of this paper.
\begin{theorem}[Error estimate for the full discretization] $ $ \label{lem:thm2} \\
Let (A1)-(A7) hold and $u$ be sufficiently smooth. 
Then 
\begin{align*}
\|u - \tilde u_h\|_{W(0,T)}
&\le
C \big( h^2 \| u\|^2_{L^2(0,T;H^2(\Omega))} + h^2 \|u'\|_{L^2(0,T;L^2(\Omega))} \\
& \qquad \qquad \qquad \qquad + \tau^2 \|u'\|^2_{L^2(0,T;H^1(\Omega))} + \tau^2 \|u''\|^2_{L^2(0,T;H^{-1}(\Omega))}\big).
\end{align*}
Like in the previous estimates, the constant $C$ here only depends on the domain, 
on the bounds for the coefficients, and on the shape regularity of the mesh. 
\end{theorem}
Let us emphasize that in view of the bounds \eqref{eq:reg1} and \eqref{eq:reg2}, 
the regularity requirements for the solution are reasonable 
and that the rates are optimal with respect to the approximation properties of the discretization. 
Also note, that no additional smoothness assumptions for the domain or on the spatial regularity of 
the parameters were required.

\section{Numerical tests} \label{sec:num}

For illustration of the benefits of our results, 
let us shortly present two irregular test problems, for 
which the standard error estimates \eqref{eq:ee1} and \eqref{eq:ee2} 
cannot be applied directly, while our estiamtes \eqref{eq:res1} and \eqref{eq:res2} 
still provide order optimal error estimates.

\subsection{Lack of smoothness}

Let us consider the one-dimensional heat equation
\begin{align*}
\partial_t u(x,t) &= u_{xx}(x,t) + f(x,t), \qquad 0<x<1, \ 0<t<1,  
\end{align*}
with homogeneous initial and boundary conditions. 
Here $\partial_t u = u'$ is used for the time derivative synonymously.
We assume that the exact solution has the form 
\begin{align*}
u(x,t) = \sum\nolimits_{n \ge 1}  u_n \sin(n \pi x) \sin(n^2 \pi t) 
\qquad \text{with} \quad u_n = (1+n^2)^{-5/4-\eps}.
\end{align*}
The parameter $\eps>0$ is assumed to be small. 
The special form of the solution allows 
to compute various norms of $u$ analytically and, in particular, to show that
\begin{align*}
u \in L^2(0,T;H^2(\Omega)) 
\qquad \text{and} \qquad
u' \in L^2(0,T;L^2(\Omega)).
\end{align*}
At the same time, one can easily check that 
\begin{align*}
u \notin L^\infty(0,T;H^{1+3\eps}(\Omega)) 
\qquad \text{and} \qquad
u' \notin L^1(0,T;H^{3\eps}(\Omega)).
\end{align*}
Due to lack of regularity, the standard a-priori estimate \eqref{eq:ee1} for the semi-discretization therefore cannot be applied,
in contrast to our energy-norm estimate \eqref{eq:res1}, which yields 
\begin{align*}
\|u - u_h\|_{L^\infty(0,T;L^2(0,\pi))} 
\le C \|u - u_h\|_{W(0,T)} \le C' h.
\end{align*}
Note that, since $u \notin L^\infty(0,T;H^{1+3\eps}(\Omega))$, 
the rate of convergence for the error in the norm of $L^\infty(0,T;L^2(\Omega))$ cannot 
be improved substantially here. 
For a solution of the form
\begin{align*}
u(x,t) = \sum\nolimits_{n \ge 1}  u_n \sin(n \pi x) \sin(n^{3/2} \pi t) 
\qquad \text{with} \quad u_n = (1+n^2)^{-5/4-\eps},
\end{align*}
one can see in a similar manner that 
\begin{align*}
u \in L^2(0,T;H^2(\Omega)), 
\qquad 
u' \in L^2(0,T;H^1(\Omega)),
\quad \text{and} \quad
u'' \in L^2(0,T;H^{-1}(\Omega)).
\end{align*}
However, $u'' \notin L^2(0,T;L^2(\Omega))$, and therefore the standard estimate 
\eqref{eq:ee2} for the full discretization cannot be applied due to lack of regularity.  
On the other hand, our error estimate  \eqref{eq:res2} for the full discretization still allows to guarantee
\begin{align*}
\|u(t^n) - u_h(t^n)\|_{L^2(\Omega)} 
\le \|u - \tilde u_h\|_{L^\infty(L^2(\Omega))}  
\le C \|u - \tilde u_h\|_{W(0,T)} \le C(h+\tau).
\end{align*}
Note that by interpolation one has at least $u \in L^\infty(0,T;H^{3/2})$ here, 
so the rate of convergence for the error $\|u(t) - u_h(t)\|_{L^2(\Omega)}$ 
in terms of the meshsize may possibly be improved. 
For sufficiently smooth solutions, the estimates of \cite{Varga,Wheeler73,Thomee} would in fact predict the optimal rate $\|u(t) - u_h(t)\|_{L^2(\Omega)} \le C (h^2+\tau)$ here.

\subsection{Discontinuous parameters}

As a second test case, we consider a thermal diffusion problem on a square covered by an inhomogeneous medium.
The governing system reads
\begin{align*}
\partial_t u(x,t) &= \div (a(x) \nabla u(x,t)) + f(x,t), \qquad x \in (-1,1)^2, \ 0<t<1/2.
\end{align*}
As before, we presecribe  homogeneous initial and boundary conditions.
Moreover, we assume that the diffusion parameter has the form 
\begin{align*}
a(x) &= \begin{cases} 1, & x_1 \cdot x_2 > 0, \\ \eps, & x_1 \cdot x_2 < 0, \end{cases}
\end{align*}
where $\eps$ is some small positive constant. It is well-known \cite{SaendigNicaise94}, that the 
associated elliptic operator 
$L u = - \div (a \nabla u)$ for such a problem is rather irregular. 
More precisely: for every $\beta>0$, one can choose an $\eps>0$, 
such that $L$ is not an isomorphism from $H^{1+\beta}_0(\Omega)$
to $H^{-1+\beta}(\Omega)$. In particular, the maximal value of $\beta$, such that 
\begin{align} \label{eq:ritzb}
\|R_h u - u\|_{L^2(\Omega)} \le C h^{\beta} \|u\|_{H^1(\Omega)}  
\end{align}
holds for arbitrary $u \in H^1_0(\Omega)$ can be made arbitrarily small. 
Note that the standard estimate \eqref{eq:ee1} does not apply directly here, since only \eqref{eq:ritzb} holds
instead of \eqref{eq:ritz}. A generalization of the error estimate to 
non-smooth problems however still allows to guarantee 
\begin{align} \label{eq:res1b}
\|u - u_h\|_{L^\infty(0,T;L^2(\Omega))} \le C h^{2\beta}, 
\end{align}
provided that the solution $u$ is sufficiently smooth; see \cite[Sec~19]{Thomee} for details.
Since $\beta$ can be aribtrarily small in general, this estimate is highly unsatisfactory. 
Our energy-norm estimate \eqref{eq:res1} for the semi-discrete 
approximation however still applies and yields 
\begin{align*}
\|u-u_h\|_{L^\infty(0,T;L^2(\Omega))} \le C \|u-u_h\|_{W(0,T)} \le C' h, 
\end{align*}
provided that the solution has the required smoothness.
This results hold regardless of the spatial regularity of the diffusion parameter $a(\cdot)$. 
For illustration of the validity of this theoretical result, we also provide some results of numerical tests. 
To verify the convergence rate for the semi-discretization, we integrated the semi-discrete problem 
numerically in time with a very accurate time stepping scheme. The resulting 
errors obtained on a sequence of uniformly refined meshes are summarized in Table~\ref{tab:1}.
\begin{table}[ht!]
\centering
\begin{tabular}{c|ccccc|c}
$h$   & 0.50000 & 0.25000 & 0.12500 & 0.06250 & 0.03125 & rate \\
 \hline
 \hline
$e_1$ & 1.65010 & 0.82062 & 0.39475 & 0.19318 & 0.09710 & 1.03   \\ 
 \hline
$e_2$ & 0.37632 & 0.11211 & 0.02928 & 0.00741 & 0.00186 & 1.93   \\ 
\end{tabular} 
\vskip1em
\caption{Errors $e_1=\|u-u_h\|_{W(0,T)}$ and $e_2=\|u-u_h\|_{L^\infty(0,T;L^2(\Omega))}$ 
obtained on a sequence of uniformly refined meshes with meshsize $h$. \label{tab:1}} 
\end{table}
\noindent
As predicted by the theory, we observe convergence of the energy-norm error 
with first order. The numerical tests actually yield a better convergence rate 
for the error in the norm of $L^\infty(0,T;L^2(\Omega))$, 
which in fact is the optimal one from an approximation point of view.
We however cannot give a full explanation for this observation yet.

\section{Summary} \label{sec:sum}

Various results concerning the numerical analysis of Galerkin approximations for parabolic problems are available in the literature; see \cite{Thomee} for a comprehensive overview and further references.
This paper contributes to this active field with providing error estimates in the energy-norm 
$\|u\|_{W(0,T)} =\|u\|_{L^2(0,T;H^1)}+\| u'\|_{L^2(0,T;H^{-1})}$, 
which seems to be a natural choice from an analytic point of view, 
but which has not been studied intensively in previous works.

In this manuscript, we considered only low order discretizations of a simple model problem. 
The general approach is however applicable to much wider class of problems and discretization schemes. 
Morover, our arguments may also be fruiteful for the derivation of a-priori error estimates in other norms \cite{DouglasDupont77,Thomee80,ThomeeWahlbin83,Wheeler73b,Wheeler75} 
and for the derivation of a-posteriori error estimates \cite{ErikssonJohnson91,ErikssonJohnson95,GeorgoulisLakkisVirtanen11,MakridakisNochetto03}.

Our main motivation to consider the energy-norm, was to overcome a sub-optimality of the 
standard estimates \eqref{eq:ee1} and \eqref{eq:ee2}
concerning the regularity requirements for the solution, for the domain, and for the parameters. 
This sub-optimality is partly due to a loose handling of time derivatives in the estimates,
and, on the other hand, stems from the use of the Ritz projector in the error decomposition,
which requires duality arguments and regularity of the underlying elliptic problem.
We therefore utilize here the $L^2$-projection in our error splitting and carefully 
estimate time derivatives in the $H^{-1}$-norm. 

In our presentation, we focused on a-priori error estimates in the energy-norm $\|\cdot\|_{W(0,T)}$, 
and we could establish optimal convergence rates under minimal 
regularity assumptions. 
By continuous embedding, we could also obtain estimates pointwise in time with the same convergence rates. 
In our numerical experiments, we observed a better convergence of the error in the norm of $L^\infty(0,T;L^2)$ for a particular problem, 
which does not follow directly from our results. 
For sufficiently regular problems, this better convergence is well explained by the standard results \cite{Thomee,Wheeler73,Varga}. 
A justification for the case of certain irregular problems is however  missing.

\section*{Acknowledgements}
The author would like to gratefully acknowledge support by the German Research Foundation (DFG) via grants GSC~233, IRTG~1529, and TRR~154.



\end{document}